\documentclass[11pt]{amsart}
\usepackage{ae,amsfonts,euscript,enumerate}
\usepackage{amsmath,euscript}
\usepackage{amssymb}
\usepackage{amsthm}
\usepackage{enumerate}
\usepackage{amsfonts}
\usepackage{comment}
\usepackage{colortbl}
\usepackage{amssymb,amsmath,array}

\newtheorem{thm}{Theorem}[section]
\newtheorem{lem}[thm]{Lemma}

\newtheorem{prop}[thm]{Proposition}

\theoremstyle{definition}

\newtheorem{hypothesis}[thm]{Hypothesis}
\newtheorem{rem}[thm]{Remark}

\newtheorem{notn}[thm]{Notation}

\newcommand{\mc}{\mathcal}

\newcommand{\spec}{{\rm Spec} \,}

\newcommand{\cha}{\operatorname{char}}

\def\co{{\mathcal O}}

\def\oqmm13{\co_q(M_{1,3})}
\def\oqm23{\co_q(M_{2,3})}

\newcommand{\mb}{\mathbb}

\usepackage{amssymb}
\usepackage{amsmath}
\usepackage[latin1]{inputenc}

\usepackage{amsmath,euscript}
\usepackage{amssymb}
\usepackage{amsthm}
\usepackage{enumerate}
\usepackage{amsfonts}
\usepackage{comment}
\usepackage{colortbl}
\usepackage{a4wide}
\usepackage{amssymb,amsmath,array}

\title[]{Free subalgebras of quotient rings of Ore extensions}
\author{Jason P.~Bell}
\thanks{The first author was supported by NSERC grant 31-611456.}
\keywords{Free algebra, division algebra, Ore extension, skew polynomial ring}

\subjclass[2000]{16K40, 16S10, 16S36, 16S85}

\address{Jason Bell\\
Department of Mathematics\\
Simon Fraser University\\
Burnaby, BC V5A 1S6\\
Canada}

\email{jpb@math.sfu.ca}

\author{D. Rogalski}
\thanks{The second author was supported by NSF grant DMS-0900981.}
\address{Daniel Rogalski\\
Department of Mathematics\\
University of California, San Diego\\
La Jolla, CA 92093-0112\\
USA}

\email{drogalsk@math.ucsd.edu}

\begin{document}
\bibliographystyle{plain}
%%%%%%%%%%%%%%%%%%%%%%%%%%%%%%%%%%%%%%%%%%%%%%%%%%%%%%%%%%%%%%%%%%%%%%%%%%%%%%%%%%%%%%%%%%%%%%%%%%%%%%%%%%%%%%%%%%%%%%

\begin{abstract}  Let $K$ be a field extension of an uncountable base field $k$, let $\sigma$ be a $k$-automorphism of $K$, and let $\delta$ be a $k$-derivation of $K$.  We show that if $D$ is one of $K(x;\sigma)$ or $K(x;\delta)$, then $D$ either contains a free algebra over $k$ on two generators, or every finitely generated subalgebra of $D$ satisfies a polynomial identity.  As a corollary, we show that the quotient division ring of any iterated Ore extension of an affine PI domain over $k$ is either again PI, or else it contains a free algebra over its center on two variables.
\end{abstract}
%%%%%%%%%%%%%%%%%%%%%%%%%%%%%%%%%%%%%%%%%%%%%5%%%%%%%%%%%%%%%%%%%%%%%%%%%%%%%%%%%%%%%%%%%%%%%%%%%%%%%%%%%%%%%%%%%%%%%
\maketitle
%\tableofcontents

\section{Introduction}

Many authors have noted that it is often the case that noncommutative division algebras have free subobjects.  For example, the existence of non-abelian free groups inside the multiplicative group $D^{\times}$ of a division algebra $D$ has been studied in several papers (see \cite{RV, Ch}, and the references therein).  It is now known that if $D$ is noncommutative and has uncountable center, then $D^{\times}$ contains a free subgroup on two generators \cite{Ch}.

The question of when a division $k$-algebra $D$ contains a free $k$-subalgebra on two generators has also attracted much attention.  The first result in this direction was obtained by Makar-Limanov \cite{ML}, who showed that if $A_1(k)=k\{x,y\}/(xy-yx-1)$ is the Weyl algebra over a field $k$ of characteristic $0$, then its quotient division algebra $D_1(k)$ does indeed contain such a free subalgebra.  This result is perhaps surprising to those only familiar with localization in the commutative setting, and is in fact a good demonstration of how noncommutative localization is less well-behaved.  In particular, the Weyl algebra $A$ is an algebra of quadratic growth; that is, if we let $V$ denote the $k$-vector subspace of $A$ spanned by $1$ and the images of $x$ and $y$ in $A$, then the dimension of $V^n$ is a quadratic function of $n$.  On the other hand, a free algebra on two generators has exponential growth.  This
is a good example of the principle that there is no nice relationship, in general, between the growth of a finitely generated algebra and the growth of other subalgebras of its quotient division algebra.

We note that by a result of Makar-Limanov and Malcolmson \cite[Lemma 1]{ML3}, if a division $k$-algebra $D$ contains a free $k$-subalgebra on two generators, then it contains a free $F$-subalgebra on two generators for any central subfield $F$.  Thus the choice of base field is not an important consideration when considering the existence of free subalgebras, and need not even be mentioned.
Now there are certain division algebras which cannot contain copies of free algebras on more than one generator for trivial reasons, for example division algebras which are algebraic over their centers.   Note also that a free algebra on two generators does not satisfy a polynomial identity.  We say that a $k$-algebra $R$ is \emph{locally PI} if every finitely generated $k$-subalgebra of $R$ is a polynomial identity ring (this is also easily seen to be independent of the choice of central base field $k$).  An obvious necessary condition for a division algebra $D$ to contain a noncommutative free algebra is that $D$ not be locally PI.  On the other hand, there are no known examples of division algebras which do not contain a free algebra on two generators, except locally PI ones.

In light of the discussion above, we say that a division algebra $D$ satisfies the \emph{free subalgebra conjecture} if $D$ contains a free subalgebra on $2$ generators if and only if $D$ is not locally PI.  In \cite{ML83}, Makar-Limanov annunciated the FOFS (full of free subobjects) conjecture, one part of which was the statement that every division algebra $D$ which is finitely generated (as a division algebra) and infinite-dimensional over its center contains a free subalgebra on 2 generators.  It is easy to see using a bit of PI theory that this statement is equivalent to what we have called the free subalgebra conjecture here.  Stafford, independently, formulated a similar conjecture \cite{Sm}.  As Makar-Limanov also notes, the conjecture is a bit provocative as stated because it implies the resolution of the Kurosh problem for division rings.  However, we will study the conjecture here only for special types of division rings in any case.

Since Makar-Limanov's original breakthrough, many authors have used his ideas to demonstrate the existence of free subalgebras on $2$ generators in the quotient division algebras of many special classes of rings, especially certain Ore extensions, group algebras, and enveloping algebras of Lie algebras \cite{FGS, Licht, Lor, ML, ML15, ML2, ML3, SG, SG98}.
\begin{comment}
For example Makar-Limanov \cite{ML2} showed that if $k$ is a field and $G$ is a nilpotent torsion-free non-abelian group then the quotient division algebra of $k[G]$ contains a free $k$-algebra on two generators.  This result was also extended to polycyclic-by-finite- that is not abelian then the quotient division algebra of the group algebra $k[G]$ contains a free $k$-algebra on two generators.  Makar-Limanov and Malcomson \cite{ML3} showed that the quotient division algebras of enveloping algebras of certain Lie algebras contain copies of a free algebra on two generators.
\end{comment}
Our main aim here is to further develop the Ore extension case.  Suppose that $D$ is a division ring with automorphism $\sigma: D \to D$ and $\sigma$-derivation $\delta$, and let $D(x; \sigma, \delta)$ be the quotient division ring of the Ore extension $D[x; \sigma, \delta]$.  Lorenz \cite{Lor} showed that $k(t)(x;\sigma)$ contains a free subalgebra on 2 generators when $\sigma$ has infinite order.  Shirvani and J. Z. Gon\c calves \cite{SG} showed that if $R$ is a $k$-algebra which is a UFD with field of fractions $K$ and the property that $R^{\times} = k^{\times}$, and $\sigma: R \to R$ is a $k$-automorphism such that the $\sigma$-fixed subring of $R$ is $k$, then $K(x; \sigma)$ contains a free subalgebra on 2 generators (in fact, even a free group algebra of rank $|k|$).

In this paper, we first give in Section~\ref{secML} some new criteria for the existence of a free $k$-subalgebra on two generators in a division ring $D(x; \sigma, \delta)$, following the main idea of Makar-Limanov's original method.  We then use these criteria to completely settle the free subalgebra conjecture for the case of Ore extensions of fields, assuming an uncountable base field.  Our main results are the following.
\begin{thm}
\label{cor: 1} \label{thm: sigma} Let $K/k$ be a field extension and $\sigma: K \to K$ a $k$-automorphism.
\begin{enumerate}
\item If $k$ is uncountable, then the following are equivalent:
\begin{enumerate}
\item[(i)] $K(x; \sigma)$ contains a free $k$-subalgebra on $2$ generators;
\item[(ii)] $K(x; \sigma)$ is not locally PI;
\item[(iii)] $K$ has an element lying on an infinite $\sigma$-orbit.
\end{enumerate}
\item If $k$ is countable, the same conclusion as in (1) holds if either $K/k$ is infinitely generated as a field extension, or if  $\sigma$ is induced by a regular $k$-automorphism of a quasi-projective $k$-variety with function field $K$.
    \end{enumerate}
\end{thm}
\noindent We expect that the free subalgebra conjecture for $K(x; \sigma)$ is always true, with no restrictions on $k$; in any case, the theorem above certainly covers the cases one is most likely to encounter.

We also study the derivation case, which is in fact easier and requires no assumption on the base field.
\begin{thm}
\label{thm: delta} Let $K$ be a field extension of a field $k$.  If $\delta: K \to K$ is a $k$-derivation, then $K(x; \delta)$ contains a free $k$-subalgebra on 2 generators  if and only if it is not locally PI.
\end{thm}
\noindent See Theorem~\ref{der-thm} for a characterization of when $K(x; \delta)$ is locally PI.

In fact, a general Ore extension $K[x; \sigma, \delta]$ of a field $K$ is isomorphic to one with either $\sigma = 1$ or with $\delta = 0$.  So as a rather quick consequence of the theorems above, we obtain the following result.
\begin{thm}
\label{thm: summary} The quotient division algebra of any iterated Ore extension of a PI domain which is affine over an uncountable field satisfies the free subalgebra conjecture.
\end{thm}

We note that our proofs are largely independent of past work in this subject, except that we assume Makar-Limanov's original result.  Some authors have considered the more general question of the existence of $k$-free group algebras in a division ring $D$, and have also studied the cardinality of the rank of the largest such free group algebra.  For simplicity, we stick to the context of the free subalgebra  conjecture here.  We mention that Shirvani and Gon\c calves have shown that if the center $k$ of $D$ is uncountable, then the existence of a free group $k$-algebra of rank $2$ in $D$ is implied by the existence of a free $k$-algebra on 2 generators \cite{SG96}.

In this paper, we have tried to make as few assumptions as possible on the ground field $k$. In forthcoming work, we will give stronger criteria for existence of free subalgebras, and thus verify the free subalgebra conjecture for some additional classes of algebras, in case $k$ is uncountable.

\section{Criteria for existence of free subalgebras of a division algebra}
\label{secML} In this section we use ideas of Makar-Limanov to give a simple criterion that guarantees that a division algebra $D$ contains a copy of a free algebra on two generators.  We work over an arbitrary field $k$.  As we noted in the introduction, the question of whether $D$ contains a free $k$-subalgebra on two generators is independent of the choice of central subfield $k$.

\begin{notn} We use the following notation: \label{notn: assume}
\begin{enumerate}
\item we let $D$ denote a division algebra over a field $k$;
\item we let $\sigma:D\rightarrow D$ denote a $k$-algebra automorphism of $D$;
\item we let $\delta:D\rightarrow D$ denote any $\sigma$-derivation of
$D$ over $k$, that is, a $k$-linear map satisfying $\delta(ab) = \sigma(a)\delta(b) + \delta(a) b$ for all $a,b  \in D$;
\item we let $\psi = (\sigma - 1) + \delta: D \to D$ (this is also a
$\sigma$-derivation) and set
$$E=\{u\in D~:~ \psi(u)= 0 \},$$ which is a division subring of $D$;
\item we let $D[x; \sigma, \delta]$ be the Ore extension
generated by $D$ and the indeterminate $x$ with relations $xa = \sigma(a)x + \delta(a)$ for $a\in D$, and let $D(x;\sigma, \delta )$ denote its quotient division algebra.  As usual, if $\sigma = 1$ we omit $\sigma$ from the notation, and if $\delta = 0$ we omit $\delta$ from the notation.
\end{enumerate}
\end{notn}

We now prove a sufficient condition for the ring $D(x;\sigma, \delta )$ to contain a free subalgebra.  Compared to the original method of Makar-Limanov's, we choose a slightly different pair of elements, and we avoid the use of power series.  We note that the characteristic of the base field has no effect in the following criterion.
\begin{thm}
\label{ml-method-thm} Assume Notation \ref{notn: assume} and let $b\in D$.   If
\begin{enumerate}
\item $b\not\in \sigma(E)$, and
\item for all $u \in D$, $\psi(u) \in \sigma(E)+ \sigma(E) b$ implies $u\in E$,
\end{enumerate}
then the $k$-algebra generated by $b(1-x)^{-1}$ and $(1-x)^{-1}$ is a free subalgebra of $D(x;\sigma, \delta)$. \label{thm: ML}
\end{thm}
\begin{proof}
Let
\begin{equation}
\mathcal{S} = \{(i_1,\ldots ,i_r)~:  r\ge 1, i_1,\ldots ,i_r\in \{0,1\}\}\cup \{\emptyset\}
\end{equation}
and for nonempty $I=(i_1,\ldots ,i_r)\in \mathcal{S}$, define ${\rm length}(I) \ := \ r$ and
\begin{equation}
W_I \ := \ b^{i_1}(1-x)^{-1}b^{i_2}(1-x)^{-1}\cdots b^{i_r}(1-x)^{-1}.
\end{equation}
If $I=\emptyset$, we define ${\rm length}(\emptyset):=0$ and $W_{\emptyset}:=1$. Note that the $W_I$ are exactly the words in the generators $b(1-x)^{-1}, (1-x)^{-1}$, and so our task is to show that $\{ W_I | I \in \mc{S} \}$ is linearly independent over $k$. It is also useful to define
\begin{equation}
V_I \ := \ (1-x)^{-1}b^{i_1}(1-x)^{-1}\cdots b^{i_r}(1-x)^{-1}
\end{equation}
for nonempty $I \in \mc{S}$, and to set $V_{\emptyset}:=(1-x)^{-1}$.

For nonempty $I = (i_1, i_2, \dots, i_r)$, we define its truncation as $I' = (i_2,\ldots ,i_r)$, with the convention that if $I$ has length $1$, then $I' = \emptyset$. Note then that trivially from the definitions we have
\begin{equation}
\label{main-VW-eq} (1-x) V_I = W_I = \ b^{i_1} V_{I'},\ \text{for}\ I \neq \emptyset.
\end{equation}

We claim that to prove that the $W_I$ are $k$-independent, it is enough to prove that $\{ V_I | I \in S \}$ is left $D$-independent.  To see this, suppose the $V_I$ are $D$-independent and that we have a nontrivial relation $\sum_{I \in S} c_I W_I = 0$ with $c_I \in k$ not all $0$.  We can assume that $c_{\emptyset} = 0$ by multiplying our relation through on the right by $(1-x)^{-1}$.  Then $0 = \sum_{I \in S} c_I W_I = \sum_{I \neq \emptyset} c_Ib^{i_1}V_{I'} = 0$.  This forces for each nonempty $I \in S$ the equation $c_{I}b^{i_1} + c_{H}b^{1-i_1} = 0$, where $H$ is the other element of $S$ which has truncation $I'$. But then $c_{I} = c_{H} = 0$ since $\{1, b \}$ is certainly $k$-independent, given that $b \not \in \sigma(E)$.  This contradicts the nontriviality of our chosen relation and the claim is proved.

The strategy will be to prove by contradiction that $\{ V_I | I \in S \}$ is left $D$-independent. In fact, it is more convenient to prove the seemingly stronger statement that this set is $D$-independent in the left factor $D$-space $D(x; \sigma, \delta)/D[x; \sigma, \delta]$.  In other words, we work modulo polynomials.  Equivalently, we suppose that we have a relation $\sum_{I \in S} \alpha_I V_I = p(x) \in D[x; \sigma, \delta]$ with $\alpha_I \in D$ not all $0$.  Among all such relations, we pick one with a minimal value of $d = \min(\operatorname{length}(I) | \alpha_I \neq 0 )$.  Moreover, among these, we select one with the smallest number of nonzero $\alpha_I$ with ${\rm length}(I)=d$. Note that certainly $d \geq 1$.
By multiplying our relation by a nonzero element of $D$, we may also assume that $\alpha_J=1$, for some $J$ of length $d$.

Now for nonempty $I$, \eqref{main-VW-eq} can be rewritten as $xV_I = V_I - b^{i_1}V_{I'}$, and for $I = \emptyset$ we have $xV_I = V_I - 1$.  Multiplying our relation on the left by $x$ and applying these formulas we obtain
\begin{gather*}
x \sum_{I \in S} \alpha_I V_I = \sum_{I \in S} [\sigma(\alpha_I) x +
\delta(\alpha_I)] V_I \\
= \sum_{I \in S} [\sigma(\alpha_I) + \delta(\alpha_I)]V_I - \sum_{I \neq \emptyset} \sigma(\alpha_I) b^{i_1} V_{I'} - \sigma(\alpha_{\emptyset}) = xp(x) \in D[x; \sigma, \delta].
\end{gather*}
Subtracting the original relation $\sum \alpha_I V_I = p(x)$, we get
\begin{equation}
\label{smaller-rel-eq} \sum_{I \in S} [\sigma(\alpha_I) - \alpha_I + \delta(\alpha_I)]V_I - \sum_{I \neq \emptyset} \sigma(\alpha_I) b^{i_1} V_{I'} = (x-1)p(x) + \sigma(\alpha_{\emptyset}) \in D[x; \sigma, \delta].
\end{equation}
But notice that since $\alpha_J = 1$, the coefficient of $V_J$ in this relation is now $0$; while no new nonzero coefficients associated to $V_I$ with $I$ of length $d$ have appeared.  Thus by our assumption that we originally picked a minimal relation, all coefficients of the $V_I$ on the left hand side of \eqref{smaller-rel-eq} are $0$. In particular, $\psi(\alpha_I) = \sigma(\alpha_I) - \alpha_I + \delta(\alpha_I) = 0$, so $\alpha_I \in E$, for all $I$ of length $d$. Also, if $H$ is the other element of $S$ with truncation $J'$, then the coefficient of $V_{J'}$ in \eqref{smaller-rel-eq} is
\[
\sigma(\alpha_{J'}) - \alpha_{J'} + \delta(\alpha_{J'}) - \sigma(\alpha_J)b^{j_1} - \sigma(\alpha_{H})b^{1-j_1} = 0.
\]
Since $H$ and $J$ have length $d$, for $u = \alpha_{J'}$ we obtain
\[
\psi(u) = \sigma(u) - u + \delta(u) \in \sigma(E)b + \sigma(E).
\]
Note that also $\psi(u) \neq 0$, as the assumption $b \not \in \sigma(E)$ implies that $\sigma(E)b + \sigma(E)$ is direct, and $\sigma(\alpha_J) = 1$.  The existence of such a $u$ violates the hypothesis, so we have achieved a contradiction. Thus $b(1-x)^{-1}$ and $(1-x)^{-1}$ generate a free subalgebra of $D(x; \sigma, \delta)$, as claimed.
\end{proof}

The interaction between $\delta$ and $\sigma$ in the criterion of the preceding theorem seems to make it hard to analyze in general. In practice, we will only use the theorem later in the special cases where $\delta = 0$ or $\sigma = 1$.  In the rest of this section, we examine the criterion for the special case of $D(x; \sigma)$ more closely.  As mentioned in the introduction, Makar-Limanov proved in \cite{ML} that the Ore quotient ring of the first Weyl Algebra, $D_1(k)$, contains a free $k$-subalgebra on two generators when $k$ has characteristic $0$ (see also Krause and Lenagan \cite[Theorem 8.17]{KL}). It is standard that $D_1(k) \cong k(u)(x; \sigma)$, where $\sigma(u) = u + 1$, but we note that Theorem~\ref{ml-method-thm}, as stated, does not recover Makar-Limanov's result.  More specifically, taking $D = k(u)$, $\sigma(u) = u+1$ and $\delta =0$ in Notation~\ref{notn: assume}, it is easy to see that $E = k$, but we have $\sigma(u) - u \in k$ with $u \not \in k$; thus the criterion in Theorem~\ref{ml-method-thm} can not be satisfied regardless of $b$.  In fact, our criterion seems to be most useful when we combine it with Makar-Limanov's known result to give the following stronger criterion.
\begin{thm}
\label{thm: criterion} \label{main-criterion-thm} Assume the notation from Notation \ref{notn: assume}, with $\delta = 0$. Suppose that either $k$ has characteristic $0$, or else $k$ has characteristic $p>0$ and we have the additional condition that $$\{a \in D | \sigma^p(a) = a \} = E.$$   If there is $b \in D \setminus E$ such that the equation
\begin{equation}
\label{key-eq-2} \sigma(u)-u \ \in \ b+E
\end{equation}
has no solutions for $u \in D$, then the $k$-algebra generated by $b(1-x)^{-1}$ and $(1-x)^{-1}$ is a free subalgebra of $D(x;\sigma)$.
\end{thm}
\noindent
\begin{proof}
Choose $b$ as in the hypothesis.  If $\sigma(u)- u \in E+Eb$ has no solutions except for $u \in E$, then we are done by Theorem~\ref{ml-method-thm}.  So we may assume there is a solution of the form $\sigma(u) - u =\alpha+\beta b$ with $\alpha,\beta\in E$ not both zero.  As long as $\beta \neq 0$, we may replace $u$ by $\beta^{-1}u$ and thus assume that $\beta=1$, and so \eqref{key-eq-2} has a solution, contradicting the hypothesis.  Thus $\beta = 0$.  Then $\alpha \not = 0$ and $y=u\alpha^{-1}$ satisfies $\sigma(y)=y+1$.  Then the elements $z = yx^{-1}$ and $x$ satisfy the relation $xz - zx = 1$.  If $\cha k = 0$, then we see that the $k$-subalgebra $R$ of $D(x;\sigma)$ generated by $x$ and $z$ is isomorphic to a factor of the Weyl algebra $A_1(k)$.  Since the Weyl algebra is simple, $R \cong A_1(k)$ and so $D(x;\sigma)$ must contain a copy of $D_1(k)$, and hence a free $k$-algebra on two generators \cite{ML}. If instead $\cha k = p > 0$, then we have $\sigma^p(y) = y$. It follows by the hypothesis that $y \in E$, but this contradicts $\sigma(y) = y + 1$.
\end{proof}

We end this section with a valuation-theoretic criterion that will be especially useful later when $D$ is a field. Recall that a \emph{discrete valuation} of a division ring $D$ is a function $\nu: D^{\times} \to \mb{Z}$ such that $\nu(xy) = \nu(x) + \nu(y)$ and $\nu(x + y) \geq \min(\nu(x), \nu(y))$ for all $x, y \in D^{\times}$.  It is easy to see that $\nu(x + y) = \min(\nu(x), \nu(y))$ if $\nu(x) \neq \nu(y)$.
The valuation $\nu$ is \emph{trivial} if $\nu(x) = 0$ for all $x \in D^{\times}$.
\begin{lem}
\label{val-lem-2} \label{lem: val}  Assume the notation from Notation \ref{notn: assume}, with $\delta = 0$.  Suppose that $D$ has a nontrivial discrete valuation $\nu: D^{\times} \to \mb{Z}$, such that (i) $\nu(a) = 0$ for all $a \in E$; and (ii) for all $a \in D^{\times}$, $\nu(\sigma^n(a)) = 0$ for all $n \gg 0$ and all $n \ll 0$.  If $\cha k = 0$, or if $\cha k = p > 0$ and $\{ y \in D | \sigma^p(y) = y \} = E$, then $D(x; \sigma)$ contains a free subalgebra on two generators.
\end{lem}
\begin{comment}
For each $n \in \mb{Z}$, define $V_n = \{a \in D | \nu(\sigma^m(a)) \geq 0\ \text{for all}\ m \geq -n \}$.  Note that setting $V = V_0$, we have $V_n = \sigma^n(V)$ for all $n$.  Since $\nu$ is nontrivial, we can pick $b$ such that $\nu(b) < 0$.  Then hypothesis (ii) implies that $b \in V_{m} \setminus V_{m+1}$ for some $m$, and it follows that $V_{n+1} \subsetneq V_n$ for all $n \in \mb{Z}$. The result now easily follows from Lemma~\ref{lem: chain}.
\end{comment}
\begin{proof}For any given $u \in D$, by hypothesis $X_u = \{n \in \mb{Z} | \nu(\sigma^n(u)) < 0 \}$ is a finite
set, and if $X_u \neq \emptyset$ we call $\ell(u) = \max X_u - \min X_u$ the \emph{length} of $u$.  If $X_u \neq \emptyset$,
then it is easy to see that $\ell(u - \sigma(u)) = \ell(u) + 1$.

By nontriviality we can pick $b \in D$ such that $X_b \neq \emptyset$.  Among all such $b$, choose one of minimal length, say $\ell(b) = d$.
We claim that there are no solutions to the equation $u - \sigma(u) = b + e$ with $u \in D$ and $e \in E$.  Suppose $u, e$ does give such
a solution.  It is easy to see that $\ell(b + e) = \ell(b) = d$, since $\nu(\sigma^n(e)) = 0$ for all $n$ by hypothesis (i).  Now $X_u = \emptyset$ is clearly impossible, so $u$ has a length.  By minimality, $\ell(u) \geq d$ and so $\ell(u - \sigma(u)) \geq d +1$, a contradiction.  The result now easily follows from Theorem~\ref{thm: criterion}.
\end{proof}

\section{The automorphism case}
\label{sec: aut}

In this section, the goal is to use the criteria developed in the previous section to study when $K(x; \sigma)$ contains a free subalgebra, where $K$ is a field, and thus to prove Theorem~\ref{thm: sigma}.  In terms of Notation~\ref{notn: assume}, we now write $D = K$ for a field $K$ containing the base field $k$ with $k$-automorphism $\sigma: K \to K$, and assume that $\delta = 0$.  Then $K(x; \sigma)$ is also an algebra over the fixed subfield $E = \{a \in K | \sigma(a) = a \}$, and as already mentioned, we may change the base field to any central subfield without affecting the question of the existence of free subalgebras.   Thus it does no harm to replace $k$ by $E$ and we assume that the base field $k$ is the $\sigma$-fixed field for the rest of this section.  We will frequently use in this section the exponent notation $b^{\sigma} := \sigma(b)$ for the action of an automorphism on an element.

The difficult direction of Theorem~\ref{thm: sigma} is to prove that $K(x; \sigma)$ contains a free subalgebra on two generators  if $K$ contains an element $a$ lying on an infinite $\sigma$-orbit.  In this case, letting $K' = k(\dots, a^{\sigma^{-2}}, a^{\sigma^{-1}}, a, a^{\sigma}, \dots)$ be the subfield of $K$ generated over $k$ by the $\sigma$-orbit of $a$, it suffices to prove that $K'(x; \sigma)$ contains a noncommutative free subalgebra.   Thus in this section we will often assume the following hypothesis.
\begin{hypothesis}
\label{main-hyp} Let $K$ be a field with automorphism $\sigma: K \to K$, and let $k$ be the fixed field of $\sigma$.  Assume that there is an element $a \in K$ on an infinite $\sigma$-orbit such that $K = k(a^{\sigma^n} | n \in \mb{Z})$.
\end{hypothesis}

The proof that $K(x; \sigma)$ satisfying Hypothesis~\ref{main-hyp} contains a free subalgebra naturally breaks up into two cases, depending on whether or not $K/k$ is finitely generated as a field extension.  The infinitely generated case is rather easily dispatched. We thank the referee very much for suggesting the elegant proof of the following proposition, which gives a simpler and more direct method for handling the infinitely generated case than our original.
\begin{prop}
\label{prop: infgen} Assume Hypothesis~\ref{main-hyp}, and suppose that $K$ is infinitely generated as a field extension of $k$.  Then $K(x; \sigma)$ contains a free $k$-subalgebra on two generators.
\end{prop}
\begin{proof}
Write $a_j = a^{\sigma^j}$ for all $j \in \mb{Z}$.  Suppose first that $k(a_i | i \geq 0)$ is still an infinitely
generated field extension of $k$.  In this case we will show that in fact the countable system of elements $\{ ax^j | j \geq 1 \}$ generates a free $k$-subalgebra of $K(x; \sigma)$.

Let $y_j  = ax^j$ for all $j \geq 1$.  Then an arbitrary monomial in the $y_j$ looks like
\[
y_{j_1} y_{j_2} \dots y_{j_n} = a_0 a_{j_1} a_{j_1 + j_2} \dots a_{j_1 + j_2 + \dots + j_{n-1}} x^{j_1 + j_2 + \dots + j_n}
\]
for some $j_1, \dots, j_n \geq 1$.
We claim that the set
\[
S = \{ a_0 a_{i_1} a_{i_2} \dots a_{i_m} | m \geq 1, 0 < i_1 < i_2 < \dots < i_m \} \cup \{ a_0 \}
\]
is linearly independent over $k$.  Suppose not, and pick a linear dependency relation over $k$ in which the maximum $n$ such that $a_n$ appears in this relation is as small as possible.  Clearly $n \geq 1$.  Since every element of $S$ is a product of distinct $a_i$, the dependency relation has the form
$p \, a_n  + q = 0$ where $p, q$ are linear combinations of elements of $S$ involving only $a_i$ with $i < n$.  If $p = 0$, we contradict the choice of $n$.  Thus $a_n = q/p \in k(a_0, a_1, \dots, a_{n-1})$.  Applying $\sigma$ this easily implies by induction that
$k(a_0, a_1, \dots, a_{n-1}) = k(a_i | i \geq 0)$, contradicting the assumption that the latter field is infinitely generated as an extension of $k$.  This establishes the claim that $S$ is linearly independent over $k$.   Together with the $K$-independence of the powers of $x$, this implies that the distinct monomials $y_{j_1} y_{j_2} \dots y_{j_n}$ are
linearly independent over $k$.  In other words, the $y_i$ generate a free subalgebra of $K(x; \sigma)$ as required.

Suppose instead that $k(a_i | i \leq 0)$ is an infinitely generated field extension of $k$.  A symmetric argument to the above shows
that $\{ ax^j | j \leq -1 \}$ generates a free subalgebra of $K(x; \sigma)$.   Finally, if both $k(a_i | i \leq 0)$ and $k(a_i | i \geq 0)$
are finitely generated field extensions of $k$, then $K = k(a_i | i \in \mb{Z})$ is also a finitely generated extension of $k$, contradicting the hypothesis.
\end{proof}

Now we begin to tackle the case where $K/k$ is a finitely generated field extension.  The idea in this case is to construct an appropriate valuation satisfying the hypothesis of Lemma~\ref{lem: val}.   The obvious valuations to use are those associated to divisors of infinite order under $\sigma$ on some variety $X$ with function field $K$.  The complication is that $\sigma: K \to K$ may correspond to a birational map $X \dashrightarrow X$ which may contract some divisors, in which case the hypotheses of Lemma~\ref{lem: val} may not hold for valuations associated to these divisors.  This problem cannot obviously be fixed by changing $X$, since there might be no variety at all with an everywhere regular automorphism corresponding to $\sigma$.  This is likely just a technical complication, and
Theorem~\ref{thm: sigma} can probably be proved for an arbitrary base field $k$ with more work.  We will finesse the issue here, by proving the main theorem with mild extra assumptions on $K/k$ which are likely to hold in practice.  We also assume for the moment that $K/k$ is totally transcendental (i.e. that every element $b \in K \setminus k$ is transcendental over $k$); this assumption will be easily removed in the proof of Theorem~\ref{thm: sigma}.
\begin{prop}
\label{prop: fingen} Assume Hypothesis~\ref{main-hyp}, and that $K/k$ is a totally transcendental finitely generated field extension of $k$.  Suppose that either $k$ is uncountable, or that there is a quasi-projective $k$-variety $X$ with function field $K$ such that $\sigma: K \to K$ is induced by a regular $k$-automorphism of $X$.  Then  $K(x; \sigma)$ contains a free subalgebra on two generators.
\end{prop}
\begin{proof}
Note that if an element $b$ in $K$ has finite order under $\sigma$, it is algebraic over $k$, thus in $k$ by the assumption that $K/k$ is totally transcendental.  Thus all elements in $K/k$ have infinite order under $\sigma$.  In particular, $K/k$ must have
transcendence degree at least $1$.

Assume first that there is a quasi-projective $k$-variety $X$ with function field $K$ such that $\sigma: K \to K$ is induced by a regular $k$-automorphism of $X$,  which we give the same name $\sigma: X \to X$ (thus on functions we have $f^{\sigma} = f \circ \sigma$).  By replacing $X$ with its normalization if necessary, we can assume that $X$ is normal.  For $f \in K$, let $(f)$ be its principal divisor on $X$ as in \cite[Section II.6]{Ha}; explicitly, $(f) = \sum_{C} v_C(f) C$ where
the sum is over all irreducible divisors $C$ on $X$,  and $v_C$ is the valuation measuring the multiplicity of the zero or pole of $f$ along $C$.  Pick $f_1 \in K \setminus k$ and suppose that $(f_1)$ has a zero or pole which is an irreducible divisor $C$ lying on an infinite $\sigma$-orbit of divisors.  Then we take the valuation $v_C$ of $K$, which satisfies the hypotheses of Lemma~\ref{lem: val} since $v_{\sigma^i(C)} = v_C \circ \sigma^i$ and any rational function has a pole along at most finitely many of the divisors $\sigma^i(C)$.  We conclude by that lemma that $K(x; \sigma)$ contains a free $k$-algebra on $2$ generators.  (Note that the extra hypothesis of Lemma~\ref{lem: val} in characteristic $p$ holds since $K/k$ is totally transcendental.)

Suppose instead that $f_1$ has all of its zeroes and poles on finite $\sigma$-orbits.   Note that $X$ certainly has infinitely many distinct irreducible divisors, since $\dim X \geq 1$.  Thus we can pick $f_2 \in K \setminus k$ such that $f_2$ has a zero or pole along some irreducible divisor not occurring among the divisors appearing in $(f_1)$.
If $(f_2)$ involves an irreducible divisor on an infinite $\sigma$-orbit, the previous paragraph applies and we are done.  Otherwise, we can choose $f_3$ such that $(f_3)$ involves at least one irreducible divisor that has not already appeared.  In this way, we either find an $f$ satisfying the previous paragraph, or else we can pick a sequence of rational functions $f_1, f_2, \dots \in K \setminus k$ such that for each $i$, $f_i$ has a zero or pole along some divisor not appearing in $(f_j)$ for all $1 \leq j < i$, and all of the irreducible divisors in $(f_i)$ lie on finite $\sigma$-orbits.
Assuming the latter case, we now apply a similar argument as in \cite[Theorem 5.7]{BRS} to show that this implies the existence of a $\sigma^n$-eigenvector in $K \setminus k$ for some $n$.  Note that for any $i$, if $n$ is a multiple of the order under $\sigma$ of all of the divisors appearing in $(f_i)$, then
$(f_i^{\sigma^{n}}) = \sigma^{-n}[(f_i)]= (f_i)$ and thus $u_i = f_i^{\sigma^n}/f_i$ is in $G = \Gamma(X, \mc{O}_X)^*$, the units group of the ring of global regular functions on $X$.  Now $H = G/(\overline{k}^* \cap G)$ is a finitely generated abelian group \cite[Lemma 5.6(2)]{BRS} which is easily seen to be torsionfree, where $\overline{k}$ is the algebraic closure of $k$.  In fact, since $K/k$ is totally transcendental, we have that $H =G/(k^* \cap G)$; say this group has rank $d$.   Then we can choose $n > 0$ such that $(f_i^{\sigma^n}) = (f_i)$ for all $1 \leq i \leq d + 1$, and so $u_i = \sigma^{n}(f_i)/f_i$ is in $G$ for all $1 \leq i \leq d + 1$.  This forces $\lambda = u_1^{a_1}u_2^{a_2} \dots u_{d+1}^{a_{d+1}} \in k$ for some integers $a_i$, not all $0$.   Then $g = f_1^{a_1}f_2^{a_2} \dots f_{d+1}^{a_{d+1}}$ satisfies $\sigma^n(g) = \lambda g$.  Moreover, $g \not \in k$, because otherwise $(f_{d+1})$ would involve only irreducible divisors occurring among the $(f_i)$ with $1 \leq i < d+1$.  Now if $\lambda$ is a root of $1$, then $\sigma^m(g) = g$ for some $m > 0$, which implies that $g$ is algebraic over $k$, contradicting that $K/k$ is totally transcendental.  So $\lambda$ has infinite multiplicative order.  Then $L = k(g)$ is a rational function field over $k$ to which the automorphism $\sigma^n$ restricts as an infinite order automorphism, and it suffices to show that $L(x; \sigma^n)$ contains a free subalgebra on two generators.  This follows from another application of Lemma~\ref{lem: val} to $L$ and its automorphism $\sigma^n \vert_L$, choosing the valuation associated to the maximal ideal $(g-1)$ of $k[g]$, which lies on an infinite $\sigma^n \vert_L$-orbit.

Next, we assume instead that $k$ is uncountable.  In this case, we will have to work with a birational map of a variety only, but
will be able to perform a similar argument to the above by choosing $f_i$ such that $(f_i)$ avoids the places where the birational map is not an isomorphism.   Since $K/k$ is finitely generated, it is well known that we can choose a normal projective $k$-variety $X$ such that $k(X) = K$.
%(Certainly we can choose a projective $k$-variety $Y$ with $K(Y) = K$.  Its normalization $X \to Y$ is a finite morphism, and finite %morphisms are projective; since compositions of projective maps are projective, $X \to \spec k$ is a projective map also.)
The automorphism $\sigma: K \to K$ corresponds to a birational map $\sigma: X \dashrightarrow X$.  Since $X$ is normal, given any irreducible divisor $C$ on $X$, $\sigma$ is defined at the generic point $\eta$ of $C$ (\cite[Lemma V.5.1]{Ha}) and so we may define $\sigma(C)$ to be the closure of $\sigma(\eta)$.  Since $\sigma$ is merely birational, $\sigma(C)$ may be a closed subset of codimension greater than $1$ in $X$, in which case we say that $\sigma$ contracts $C$.  However, since $\sigma$ is an isomorphism on some open subset of $X$,  $\sigma$ must contract at most finitely many irreducible divisors.  Then it is clear that the set $S$ of irreducible divisors which are contracted by some $\sigma^n$ with $n \in \mb{Z}$ is countable.

We now show that we can find a plentiful supply of rational functions whose $\sigma$-iterates have divisors entirely avoiding the bad set $S$.  Pick any element $h \in K \setminus k$.  Consider $h + \lambda$ as $\lambda \in k$ varies.  The divisors along which $h + \lambda_1$ and $h + \lambda_2$ have a zero are disjoint if $\lambda_1 \neq \lambda_2$.  For a given $i \in \mb{Z}$, there are countably many $\lambda$ such that $h^{\sigma^i} + \lambda$ has a zero along a divisor in $S$.  Since $k$ is uncountable, there are uncountably many $\lambda$ such that $h^{\sigma^i} + \lambda$ has zeroes only along divisors not in $S$, for all $i$.  Fix such a $\lambda$ and put $g = 1/(h +\lambda)$; thus $g^{\sigma^i}$ has poles only along divisors not in $S$, for all $i$.  By the same argument, for uncountably many $\mu$ the rational functions $g^{\sigma^i} + \mu$ have no zeroes in $S$, for all $i$.  Let 
$f_{\mu} = g + \mu$.  By construction, there are uncountably many $\mu \in k$ such that $f_{\mu}^{\sigma^i}$ has no zeroes or poles 
in $S$, for all $i$; and moreover, $f_{\mu_1}$ and $f_{\mu_2}$ have disjoint zeroes if $\mu_1 \neq \mu_2$.  

Now notice that if $C$ is an irreducible divisor not in $S$, so that $\sigma^i(C)$ is again an irreducible divisor, then we will still have $v_{\sigma^i(C)} = v_C \circ \sigma^i$ for the associated valuations, as in the automorphism case; this is because $\sigma^i$ will give an isomorphism locally from the generic point of $C$ to the generic point of $\sigma^i(C)$.  Moreover, if $f \in K$ has the property that $D = (f)$ and $E = (f^{\sigma^{-i}})$ both involve no divisor in $S$, then we must have $E = \sigma^{i}(D)$, since $\sigma^i$ will give a bijection from the generic points of the divisors involved in $D$ to those involved in $E$.  Since these are the only properties that were really needed in paragraphs 2 and 3 of the proof, now repeat the argument of those paragraphs, choosing the $f_i$ only from among the uncountably many good $f_{\mu}$'s.  Note that by construction, given any list of such, say $f_1, f_2, \dots f_{i-1}$, we can certainly find such an $f_i$ which has a zero along some divisor not appearing in $(f_j)$ for all $1 \leq j < i$.
Thus the argument in paragraphs 2 and 3 constructs a free subalgebra of $K(x; \sigma)$ in this case also.  
\end{proof}

\begin{proof}[Proof of Theorem~\ref{thm: sigma}]
Let $\sigma:K \to K$ be an automorphism of $K/k$.  If every element of $K$ lies on a finite $\sigma$-orbit, then setting $K_n = \{ x \in K | \sigma^n(x) = x \}$, we will have $K = \bigcup_{n \geq 1} K_n$ and thus $K(x; \sigma) = \bigcup_{n \geq 1} K_n(x; \sigma)$ is a directed union of PI algebras.  Thus it is locally PI, and cannot possibly contain a free subalgebra on two generators.

To complete the proof, we assume that there is an element $a \in K$ lying on an infinite $\sigma$-orbit, and need to prove that $K(x; \sigma)$ contains a free subalgebra.  We have seen that we may assume the conditions in Hypothesis~\ref{main-hyp}, so that $k$ is the fixed field of $\sigma$ and $K = k(a^{\sigma^n} | n \in \mb{Z})$.  If $K/k$ is infinitely generated as a field extension, then
we are done by Proposition~\ref{prop: infgen}, with no assumptions on the base field $k$ necessary.  Suppose instead that $K/k$
is finitely generated, and that we have either that (i) $k$ is uncountable, or (ii) there is a $k$-automorphism $\sigma$ of a quasi-projective $k$-variety $X$ with $k(X) = K$ inducing $\sigma: K \to K$ (we may assume that $X$ is normal).   To apply Proposition~\ref{prop: fingen} we need to reduce to the totally transcendental case.   If $L \subseteq K$ is the
subfield of elements algebraic over $k$, then $L/k$ is also finitely generated, and thus $[L: k] <\infty$. The elements in $L$ have finite order under $\sigma$, and thus there is a single power $\sigma^d$ such that $\sigma^d(b) = b$ for all $b \in L$. Let $K' := k(a^{\sigma^{nd}} | n \in \mb{Z})$.  We now replace $(K, \sigma)$ by $(K', \sigma^d)$.   By construction, $K'$ is still generated by the $\sigma^d$-iterates of a single element $a$ on an infinite orbit.  The field  $k' = \{b \in K | \sigma^d(b) =b \}$ certainly contains $L$; in fact, $k'$ is algebraic over the field $k$ of $\sigma$-fixed elements, and so $k' = L$.  Thus the fixed field of $\sigma^d: K' \to K'$ is now $K' \cap L$, and the field extension $K'/(K' \cap L)$ is now totally transcendental. If we have hypothesis (ii), then $X$ is an $L$-variety since rational functions which are
algebraic over $k$ must be global regular functions ($X$ is normal), and  $\sigma^d: X \to X$ is now an $L$-automorphism.  Thus in either case (i) or (ii), replacing the triple $(K/k, \sigma, a)$ with $(K'/(K' \cap L), \sigma^d, a)$ preserves the hypothesis, and now $K'/(K' \cap L)$ is totally transcendental.  By Proposition~\ref{prop: fingen}, $K'(x; \sigma^d)$ contains a free $k$-subalgebra on $2$ generators, and
thus so does the larger division algebra  $K(x; \sigma)$.
\end{proof}

\begin{rem}
We note that the division rings $K(x; \sigma)$ really can be only
locally PI rather than PI.  For example, let $K = \mb{C}(y_1, y_2,
\dots )$ be a function field in infinitely many indeterminates and
define an automorphism $\sigma: K \to K$ which fixes $\mb{C}$ and
has $\sigma(y_n) = \zeta_n y_n$ for a primitive $n$th root of unity
$\zeta_n$.  If $K_n = \mb{C}(y_n)$ then it is easy to check that
$K_n(x; \sigma)$ is a subdivision ring of $K(x; \sigma)$ with PI
degree exactly $n$.
\end{rem}

\section{The derivation case}

Assume Notation~\ref{notn: assume}, with  $K = D$ be a field extension of $k$, $\sigma = 1$, and $\delta: K \to K$ a $k$-derivation, so $E = \{ a \in K | \delta(a) = 0 \}$.  In this section we show that $K(x; \delta)$ satisfies the free subalgebra conjecture. Since $K(x; \delta)$ is an $E$-algebra, as usual we can and do replace the base field $k$ by the central subfield $E$.   In fact, the analysis of $K(x; \delta)$ is much easier than the automorphism case.  In characteristic $0$, this reduces rather trivially to the case of the Weyl algebra (as other authors have also observed).  So our main contribution here is to consider the characteristic $p$ case.
\begin{thm}
\label{der-thm} Let $\delta: K \to K$ be a derivation of a field, where $k = \{a \in K | \delta(a) = 0\}$.
\begin{enumerate}
\item If $\cha k = 0$, then $K(x; \delta)$ contains a free
subalgebra if and only if $\delta \neq 0$.
\item If $\cha k = p > 0$, then $K(x; \delta)$ contains a free
subalgebra if and only if there is an element $a \in K$ such that setting $F_i = k(a, \delta(a), \dots, \delta^{i-1}(a))$, one has $F_i \subsetneq F_{i+1}$ for all $i \geq 0$.
\end{enumerate}
\end{thm}
\begin{proof}
(1). If $\delta \neq 0$, say $\delta(a) \neq 0$, take $y = a$ and $z = x(\delta(a))^{-1}$.  Then $yz - zy =1$, so $D$ contains a copy of the Weyl algebra $A_1(k)$, and hence a free algebra on two generators by Makar-Limanov's original result \cite{ML}. On the other hand, if $\delta = 0$ then $K(x; \delta) \cong K(x)$ is commutative and cannot contain a noncommutative free subalgebra.

(2). Since $\cha k = p$, note that $\delta(b^p) = 0$ for all $b \in K$, so $k$ contains all $p$th powers.  Now fix $a \in K$ and consider the fields $F_i = k(a, \delta(a), \dots, \delta^{i-1}(a))$.  For each $i \geq 1$, if $F_{i-1} \subsetneq F_i$, then since $[\delta^{i-1}(a)]^p \in k \subseteq F_{i-1}$, we must have $[F_i:F_{i-1}] = p$ and $F_{i-1} \subseteq F_i$ is a purely inseparable simple extension.

Suppose that $F_{i-1} \subsetneq F_i$ for all $i \geq 1$, let $F = \bigcup_{i \geq 0} F_i$, and note that $F$ is closed under $\delta$; so it is enough to prove that $F(x; \delta)$ contains a free subalgebra.  Write $b_i = \delta^i(a)$.  By the analysis of the previous paragraph, it easily follows that $F$ has a $k$-basis consisting of all words in the $b_i$ of the form $\{ b_0^{e_1} b_1^{e_2} \dots b_m^{e_m} | 0 \leq e_j \leq p-1 \}$. Now we apply the criterion of Theorem~\ref{ml-method-thm}, with the choice $b = b_0 = a$. Thus it is sufficient to prove the claim that if $u\in F$ satisfies $\delta(u)\in k+kb$, then $u \in k$. To obtain this claim, suppose that $u$ satisfies $\delta(u) \in k + kb$ with $u \not \in k$, so there exists some $d\ge 0$ such that $u \in F_{d+1} \setminus F_d$.  We can write $u$ as
$$u = \sum_{i=0}^{p-1} u_i b_d^i,$$
where each $u_i\in F_d$ and $u_i \neq 0$ for some $i>0$. Thus we see that
$$\delta(u) =  \sum_{i=0}^{p-1} i u_i  b_d^{i-1} b_{d+1} +
\sum_{i = 0}^{p-1} \delta(u_i) b_d^i, $$ where the second sum is contained in $F_{d+1}$.  Since some $u_i \not = 0$ with $i \neq 0$, we have $\delta(u) \not \in F_{d+1}$, contradicting the assumption $\delta(u) \in k + kb$.  This proves the claim, and so $K(x; \delta)$ contains a free algebra on two generators.

On the other hand, suppose that for some $a \in K$, the sequence of fields $F_i = k(a, \delta(a), \dots, \delta^{i-1}(a))$ has $F_i = F_{i+1}$ for some $i$.  Then it is easy to see that $F_n = F_i$ for all $n \geq i$, and so $F_i$ is a $\delta$-invariant subfield, of finite degree over $k$.  If this happens for every $a \in K$, then every finite subset of $K$ is contained in a $\delta$-invariant subfield $F$ of finite degree over $k$, and so is contained in the PI division ring $F(x; \delta)$.   Thus $K(x; \delta)$ is locally PI and does not contain a noncommutative free subalgebra.
\end{proof}

\begin{proof}[Proof of Theorem~\ref{thm: delta}]
Examining the proofs of parts (1) and (2) of Theorem~\ref{der-thm}, we see that $K(x; \delta)$ contains a free subalgebra if and only if it is not locally PI.
\end{proof}

An interesting example of part (2) of Theorem~\ref{der-thm} is obtained by taking $K = \mb{F}_p(x_0, x_1, \dots )$ to be a rational function field in infinitely many indeterminates over the field of $p$ elements, and defining $\delta(x_i) = x_{i+1}$ for all $i \geq 0$. The ring $K(x; \delta)$ then contains a free algebra in two generators over $\mb{F}_p$.  This ring has appeared before in the literature and has other interesting properties.  In particular, Resco and Small studied this ring in \cite{ReSm} as an example of a noetherian affine algebra which becomes non-noetherian after base field extension.

\section{Summary theorems}

In this section, we apply our results to show that the free subalgebra conjecture holds for a large class of algebras formed from iterated Ore extensions.  We state our summary theorems over an uncountable field for convenience, though they hold over an arbitrary field whenever the iterated Ore extension is built out of extensions satisfying Theorem~\ref{thm: sigma}(2).

Before proving our main theorem, we make an easy observation. The reason that we have not yet considered Ore extensions with both an automorphism and derivation is the following fact.
\begin{lem}
\label{delete-lem} Let $D$ be a PI division algebra with automorphism $\sigma$ and $\sigma$-derivation $\delta$.  Then $D[x; \sigma, \delta]$ is isomorphic either to $D[x'; \sigma']$ for some other automorphism $\sigma'$, or else to $D[x'; \delta']$ for some derivation $\delta'$.
\end{lem}
\begin{proof}
This is presumably well-known, but we sketch the proof since it is elementary. Let $Z = Z(D)$. For any $a \in D$, $b \in Z$, we have $\delta(ab) = \delta(ba)$ and so
\begin{equation}
\label{commute-eq} \sigma(a) \delta(b) + \delta(a) b = \sigma(b) \delta(a) + \delta(b) a.
\end{equation}
Since $\sigma(b) \in Z$ also, we have $\delta(a)[b - \sigma(b)] = \delta(b) a - \sigma(a) \delta(b)$.  Then if there is any $b \in Z$ such that $\sigma(b) \neq b$, we must have $\delta(a) = (b - \sigma(b))^{-1}[\delta(b) a - \sigma(a) \delta(b)]$ for all $a \in D$.  In this case, making the change of variable $x' = x + (b - \sigma(b))^{-1}\delta(b)$, one easily checks that $D[x; \sigma, \delta] \cong D[x'; \sigma, 0]$.

Otherwise $\sigma(b) = b$ for all $b \in Z$, in other words $\sigma$ is trivial on the center.  By the Skolem-Noether theorem, $\sigma$ is an inner automorphism of $D$, say $\sigma(a) = d^{-1}a d$ for all $a$, some $d \in D^{\times}$.  Then the change of variable $x' =d x$ gives $D[x; \sigma, \delta] \cong D[x'; 1,  d\delta]$.
\end{proof}

\begin{thm}
\label{sum-thm} Let $k$ be an uncountable field.  The following results hold:
\begin{enumerate}
\item Let $A$ be any PI domain which is a $k$-algebra with automorphism $\sigma$ and $\sigma$-derivation $\delta$ (over $k$).  Then the quotient division algebra of $A[x; \sigma, \delta]$ satisfies the free subalgebra conjecture.
\item If $A$ is any affine $k$-algebra which is an Ore domain such that $Q(A)$ satisfies the free subalgebra conjecture, then $Q(A[x; \sigma, \delta])$ also satisfies the conjecture.
\end{enumerate}
\end{thm}
\begin{proof}
(1). Let $D$ be the quotient division algebra of $A$, so that $R$ has quotient ring $Q(R) = D(x; \sigma, \delta)$.  By Lemma~\ref{delete-lem}, it is enough to consider the two special cases $D(x; \sigma)$ and $D(x; \delta)$.  If $K = Z(D)$, then $\sigma$ restricts to $K$ and $D$ is finite over $K$ since $D$ is PI.  It is easy to see that it is enough to prove the free subalgebra conjecture for $K(x; \sigma)$. Now Theorem~\ref{thm: sigma} gives the result.

Similarly, considering $D(x; \delta)$, we have $\delta(K) \subseteq K$ (use \eqref{commute-eq}) and so we easily reduce to the case of $K(x; \delta)$.  We are done by Theorem~\ref{thm: delta}.

(2). If $A$ is not locally PI, then by assumption $Q(A)$ contains a free subalgebra on $2$ generators.  Then there is an embedding $Q(A) \subseteq Q(A[x; \sigma, \delta])$ and of course $Q(A[x; \sigma, \delta])$ is also not locally PI, so we are done in this case. If instead $A$ is locally PI, then it is actually PI by the assumption that $A$ is affine.  Now part (1) applies.
\end{proof}

\begin{proof}[Proof of Theorem~\ref{thm: summary}]
An easy induction using parts (1) and (2) of Theorem~\ref{sum-thm} shows that any iterated Ore extension of an affine PI domain over an uncountable field has a quotient division algebra which satisfies the free subalgebra conjecture.
\end{proof}

\section*{Acknowledgments}
We thank James Zhang, Sue Sierra, Agata Smoktunowicz, Toby Stafford, and Lance Small for many valuable discussions.  We thank Jairo Gon\c{c}alves for pointing out an error in a previous version of the paper. Finally, we are grateful to the referee for providing the proof of Proposition~\ref{prop: infgen}.


\begin{thebibliography}{99}
\begin{comment}
\bibitem{AS} M. Artin and J. T. Stafford,
Noncommutative graded domains with quadratic growth. \emph{Invent. Math.} {\bf 122} (1995), no. 2, 231--276.
\bibitem{Bell1} J. P. Bell, Division algebras of Gelfand-Kirillov transcendence degree 2. \emph{Israel J. Math.} {\bf 171} (2009), 51--60.
\bibitem{Bell2} J. P. Bell, Centralizers in domains of finite Gelfand-Kirillov dimension. \emph{Bull. Lond. Math. Soc.} {\bf 41} (2009), no. 3, 559--562.
\end{comment}
\bibitem{BRS} J. P. Bell, D. Rogalski, and S. Sierra, The Dixmier-Moeglin equivalence for twisted homogeneous coordinate rings. \emph{Israel J. Math.} {\bf 180} (2010), no. 1, 461--507.
\begin{comment}
\bibitem{BG} K. Brown and K. Goodearl, \emph{Lectures on algebraic quantum groups.} Advanced Courses in Mathematics. CRM Barcelona. Birkh\"auser Verlag, Basel, 2002.
\end{comment}

\bibitem{Ch} K. Chiba, Free subgroups and free subsemigroups of division rings. \emph{J. Algebra} {\bf 184} (2) (1996) 570--574.
\bibitem{FGS} L. M. Figuerido, J. Z. Gon\c calves, M. Shirvani, Free group algebras in certain division rings. \emph{J. Algebra} {\bf 185} (2) (1996)
298--311.

\bibitem{SG96}
J.~Goncalves and M.~Shirvani, On free group algebras in division rings
  with uncountable center. \emph{Proc. Amer. Math. Soc.} \textbf{124} (1996), no.~3,
  685--687.
  % \MR{1291771 (96f:16020)}

\bibitem{Ha}
Robin Hartshorne, \emph{Algebraic geometry}, Springer-Verlag, New York, 1977,
  Graduate Texts in Mathematics, No. 52.% \MR{57 \#3116}
\bibitem{KL} G. R. Krause, and T. H. Lenagan, \emph{Growth of algebras and Gelfand-Kirillov dimension.} Revised
edition. Graduate Studies in Mathematics, 22. American Mathematical Society, Providence, RI, 2000.
\bibitem{Licht} A. I. Lichtman, Free subalgebras in division rings generated by universal enveloping algebras. \emph{Algebra Colloq.} {\bf 6} (2)
(1999) 145--153.
\bibitem{Lor} M. Lorenz, On free subalgebras of certain division algebras. \emph{Proc. Amer. Math. Soc.} {\bf 98} (3) (1986) 401--405.
\bibitem{ML} L. Makar-Limanov,
The skew field of fractions of the Weyl algebra contains a free noncommutative subalgebra. \emph{Comm. Algebra} {\bf 11} (1983), no. 17, 2003--2006.
\bibitem{ML83}
L. Makar-Limanov, On free subobjects of skew fields.  \emph{Methods in ring
  theory ({A}ntwerp, 1983)}, NATO Adv. Sci. Inst. Ser. C Math. Phys. Sci., vol.
  129, Reidel, Dordrecht, 1984, pp.~281--285.
\bibitem{ML15} L. Makar-Limanov, On free subsemigroups of skew fields. \emph{Proc. Amer. Math. Soc.} {\bf 91} (2) (1984) 189--191.
\bibitem{ML2} L. Makar-Limanov,
On group rings of nilpotent groups. \emph{Israel J. Math.} {\bf 48} (1984), no. 2-3, 244--248.
\bibitem{ML3} L. Makar-Limanov and P. Malcolmson, Free subalgebras of enveloping fields. \emph{Proc. Amer. Math. Soc.} {\bf 111} (1991), no. 2, 315--322.
\begin{comment}
%\bibitem{MR} J. C. McConnell and J. C. Robson, \emph{Noncommutative Noetherian rings.}
%With the cooperation of L. W. Small. Pure and Applied Mathematics (New York). A Wiley-Interscience Publication. John Wiley \& Sons, Ltd., %Chichester, 1987.
%\bibitem{R} Z. Reichstein, On a question of Makar-Limanov. \emph{Proc. Amer. Math. Soc.} {\bf 124} (1) (1996) 17--19.
\end{comment}

\bibitem{RV} Z. Reichstein and N. Vonessen, Free subgroups in division algebras. \emph{Comm. Algebra} {\bf 23} (6) (1995) 2181--2185.

\bibitem{ReSm}
Richard Resco and L.~W. Small, Affine {N}oetherian algebras and
  extensions of the base field. \emph{Bull. London Math. Soc.} \textbf{25} (1993),
  no.~6, 549--552. %\MR{1245080 (94m:16029)}

\bibitem{SG} M. Shirvani and J. Z. Gon\c calves,
Large free algebras in the ring of fractions of skew polynomial rings. \emph{J. London Math. Soc.} (2) {\bf 60} (1999), no. 2, 481--489.

\bibitem{SG98}
M.~Shirvani and J.~Z. Gon{\c{c}}alves, Free group algebras in the field
  of fractions of differential polynomial rings and enveloping algebras. \emph{J.
  Algebra} \textbf{204} (1998), no.~2, 372--385.% \MR{1624443 (99g:16033)}

\bibitem{Sm} L. W. Small, private communication.

\begin{comment}

%\bibitem{Smok2} A. Smoktunowicz, On structure of domains with quadratic growth.
%\emph{J. Algebra} {\bf 289} (2005), no. 2, 365--379.
%\bibitem{Smok1} A. Smoktunowicz, Makar-Limanov's conjecture on free subalgebras. \emph{Adv. Math.} {\bf 222} (2009), no. 6, 2107--2116.
%\bibitem{Smok3} A. Smoktunowicz,There are no graded domains with GK dimension strictly between 2 and 3. \emph{Invent. Math.} {\bf 164} %(2006), no. 3, 635--640.
%\bibitem{Z} J. J. Zhang, On lower transcendence degree. \emph{Adv. Math.} {\bf 139} (1998), no. 2, 157--193.

\end{comment}
\end{thebibliography}
\end{document}